\documentclass{amsart}

\newtheorem{theorem}{Theorem}[section]
\newtheorem{lemma}[theorem]{Lemma}
\newtheorem{corollary}[theorem]{Corollary}

\theoremstyle{definition}

\theoremstyle{remark}
\newtheorem{remark}[theorem]{Remark}
\usepackage{amsmath}
\usepackage{amssymb}
\usepackage{graphicx}
\usepackage{mathrsfs}
\usepackage{algorithm}
\usepackage{algorithmicx}
\usepackage{appendix}
\usepackage{enumitem}
\usepackage{stmaryrd}
\usepackage{multirow}
\usepackage[colorlinks, linkcolor=blue, anchorcolor=blue, citecolor=blue]{hyperref}
\AtBeginDocument{
                 \label{CorrectFirstPageLabel}
                 
                }

\DeclareMathOperator*{\rank}{rank}
\DeclareMathOperator*{\diag}{diag}

\DeclareMathOperator*{\Span}{span}
\let\s=\scriptscriptstyle

\numberwithin{equation}{section}

\begin{document}

\title[TG for nonsymmetric positive definite systems]{Convergence analysis of two-grid methods for nonsymmetric positive definite systems}

\author{Xuefeng Xu}
\address{School of Mathematics, Southeast University, Nanjing 211189, China}
\email{xuxuefeng@lsec.cc.ac.cn; xuxuefeng@seu.edu.cn}





\begin{abstract}
The convergence theory of multigrid methods for symmetric positive definite systems is well established. For nonsymmetric systems, however, the corresponding theory remains far from mature. Two-grid analysis is fundamental to the design and analysis of multigrid methods. This paper presents a convergence analysis of two-grid methods for nonsymmetric positive definite systems. When the coarse-grid system is solved exactly, we derive a succinct identity for the two-grid convergence factor measured in a smoother-induced norm. More generally, under mild assumptions, we develop a convergence theory for inexact two-grid methods, where convergence is measured in a generic norm.
\end{abstract}

\maketitle



\section{Introduction}

For a large class of linear systems arising from the discretization of partial differential equations, multigrid is a powerful solver with linear or near-linear computational complexity; see, e.g.,~\cite{Hackbusch1985,Briggs2000,Trottenberg2001,Vassilevski2008,XZ2017}. The foundation of multigrid is a two-grid scheme, consisting of two complementary error-reduction processes: \textit{smoothing} (or \textit{relaxation}) and \textit{coarse-grid correction}. Typically, the smoothing process is chosen as a simple iterative method, such as weighted Jacobi or Gauss--Seidel. However, these classical methods are generally effective only at reducing high-frequency (or oscillatory) error modes. The remaining low-frequency (or smooth) modes are then handled by the coarse-grid correction. These two processes are applied iteratively until a desired tolerance is reached.

Most existing multigrid theories require the system matrix to be symmetric positive definite (SPD). For such systems, the convergence theory of multigrid methods is well established in the literature; see, e.g.,~\cite{Xu1992,XZ2002,Falgout2004,Falgout2005,MacLachlan2014,Notay2015,XZ2017,XXF2022-1,XXF2022-2,XXF2025-1,XXF2025-2}. When the coarse-grid system is solved exactly, the energy norm of the two-grid iteration matrix, commonly termed the \textit{convergence factor}, coincides with its spectral radius, provided that the pre- and postsmoothing steps are performed symmetrically. This convergence factor admits a characterization via an identity~\cite{XZ2002,Falgout2005,Zikatanov2008}, which has been widely used in the analysis of two-grid methods; see, e.g.,~\cite{Falgout2005,Brannick2018,XXF2018}. In practice, however, it is often too costly to solve the coarse-grid system exactly. A standard remedy is to apply the two-grid scheme recursively within the correction step, yielding a multigrid method with an inherent multilevel hierarchy. Thus, multigrid can be interpreted as an inexact two-grid scheme. For inexact two-grid theory and its multilevel extensions, we refer the reader to~\cite{XXF2022-1,XXF2022-2,XXF2025-1,XXF2025-2} and the references therein.

Unlike in the SPD setting, multigrid theory for nonsymmetric systems remains far from mature. Most existing algorithms are either heuristic-based or built upon incomplete theoretical foundations. In the nonsymmetric setting, the system matrix often fails to induce a natural norm, and the error propagation matrix of the exact coarse-grid correction is generally an \textit{oblique} projector. The correction process is a primary mechanism for error reduction; in the nonsymmetric case, however, it may actually amplify the error. This renders the development of a complete multigrid theory particularly challenging.

Two-grid analysis plays a fundamental role in the design and analysis of multigrid methods. In the nonsymmetric setting, two key questions arise, as outlined below.

\begin{itemize}[leftmargin=0.8cm]

\item The first question concerns the choice of an appropriate convergence measure. A natural candidate is the spectral radius, but it captures only the asymptotic behavior. That is, the spectral radius may be an unreliable convergence indicator in the non-asymptotic regime. Alternative measures and their associated convergence estimates have been discussed in~\cite{Brezina2010,Lottes2017,Manteuffel2017,Manteuffel2019-1,Manteuffel2019-2,Notay2020,Southworth2024,Krzysik2026} and the references therein. However, some of these measures---such as the $\sqrt{A^{T}A}$- and $\sqrt{AA^{T}}$-norms, where $A$ denotes the system matrix---may be difficult to compute or interpret in practice.

\smallskip

\item The second question concerns the construction of restriction and prolongation matrices (denoted by $R$ and $P$, respectively) such that the correction process does not amplify the error. Assume that $RAP$ is nonsingular, and define
\begin{equation}\label{pi}
\Pi:=P(RAP)^{-1}RA.
\end{equation}
The error propagation matrix of the exact coarse-grid correction is then $I-\Pi$, which is a nonzero projector. For any norm $\|\cdot\|_{X}$ (induced by an SPD matrix $X$), it holds that
\begin{displaymath}
\|I-\Pi\|_{X}\geq 1.
\end{displaymath}
To ensure that the correction process does not amplify the error, it suffices to choose $R$ and $P$ such that
\begin{equation}\label{ortho-proj}
\|I-\Pi\|_{X}=1.
\end{equation}
Such $R$ and $P$ are referred to as \textit{compatible transfer operators} in~\cite{Southworth2024}.

\end{itemize}

This paper presents a convergence analysis of two-grid methods for nonsymmetric positive definite systems. Our main contributions are twofold.

\begin{itemize}[leftmargin=0.8cm]

\item When the coarse-grid system is solved exactly, we use an SPD smoother $M$ and adopt its induced norm as the convergence measure. As shown in Remark~\ref{rmk:proj}, \eqref{ortho-proj} with $X=M$ holds if and only if
\begin{displaymath}
\mathcal{R}(P)=\mathcal{R}(M^{-1}A^{T}R^{T}),
\end{displaymath}
where $\mathcal{R}(\cdot)$ denotes the column space of a matrix. For a given restriction matrix $R$, we take $M^{-1}A^{T}R^{T}$ as the prolongation matrix, which can be viewed as a multiplicative perturbation of the classical choice $R^{T}$. Under these settings, we derive an identity for the two-grid convergence factor, which can be used to analyze the optimality of $R$ and the influence of $\mathcal{R}(R^{T})$ on the convergence factor.

\smallskip

\item When the coarse-grid system is solved approximately, we take the generic $X$-norm as the convergence measure and $X^{-1}A^{T}R^{T}$ as the prolongation matrix. Under these settings, we show that if the smoothing process does not amplify the error and the exact two-grid method is convergent, then for any convergent coarse solver, the resulting inexact method is also convergent. Moreover, the uniform convergence of the exact two-grid method carries over to its inexact variant, provided that the relative accuracy parameter of the coarse solver is fixed; see Remark~\ref{rmk:uniform} for details.

\end{itemize}

In addition to the theoretical analysis, we provide a numerical example based on the discretization of the convection-diffusion equation to compare the performance of $R^{T}$ and an approximation of $X^{-1}A^{T}R^{T}$ with $X=(A+A^{T})/2$. Numerical results show that the latter prolongation matrix yields smaller two-grid convergence factors in the convection-dominated regime and leads to a more robust two-grid method.

The rest of this paper is organized as follows. In Section~\ref{sec:pre}, we give a fundamental analysis of two-grid methods. In Section~\ref{sec:TG}, we present an identity for the exact two-grid convergence factor, followed by two applications. In Section~\ref{sec:iTG}, we establish a convergence theory for two-grid methods with inexact coarse solvers. In Section~\ref{sec:numer}, we compare the performance of the classical prolongation matrix $R^{T}$ with that of the new one motivated by our theory via a numerical example. Finally, in Section~\ref{sec:con}, we give some concluding remarks.

\section{Preliminaries} \label{sec:pre}

We first introduce some notation used throughout the paper.

\begin{itemize}[leftmargin=1.85cm]

\item[$I_{n}$:] the $n\times n$ identity matrix (or $I$ when the size is clear from the context).

\item[$\mathcal{R}(\cdot)$:] the column space (or range) of a matrix.

\item[$\mathcal{N}(\cdot)$:] the null space (or kernel) of a matrix.

\item[$\lambda_{\min}(\cdot)$:] the smallest eigenvalue of a matrix with real eigenvalues.

\item[$\lambda_{\max}(\cdot)$:] the largest eigenvalue of a matrix with real eigenvalues.

\item[$\lambda_{i}(\cdot)$:] the $i$th smallest eigenvalue of a matrix with real eigenvalues.

\item[$\lambda(\cdot)$:] the spectrum of a matrix.

\item[$\|\cdot\|_{2}$:] the spectral norm of a matrix.

\item[$\|\cdot\|_{X}$:] the norm induced by an SPD matrix $X\in\mathbb{R}^{n\times n}$: $\|\mathbf{v}\|_{X}:=\sqrt{\mathbf{v}^{T}X\mathbf{v}}$ for $\mathbf{v}\in\mathbb{R}^{n}$; $\|Y\|_{X}:=\max_{\mathbf{v}\in\mathbb{R}^{n}\backslash\{0\}}\frac{\|Y\mathbf{v}\|_{X}}{\|\mathbf{v}\|_{X}}$ for $Y\in\mathbb{R}^{n\times n}$.

\end{itemize}

Let $A\in\mathbb{R}^{n\times n}$ be nonsymmetric but positive definite (that is, $\mathbf{v}^{T}A\mathbf{v}>0$ for all $\mathbf{v}\in\mathbb{R}^{n}\backslash\{0\}$), and let $\mathbf{f}\in\mathbb{R}^{n}$. Consider solving the linear system
\begin{equation}\label{system}
A\mathbf{u}=\mathbf{f}.
\end{equation}
It is straightforward to verify that $A$ is positive definite if and only if its symmetric part, $\frac{1}{2}(A+A^{T})$, is positive definite.

\begin{remark}
Let
\begin{displaymath}
A=D+L+U,
\end{displaymath}
where $D$, $L$, and $U$ denote the diagonal, strictly lower triangular, and strictly upper triangular parts of $A$, respectively. It is easy to see that $D$, $(D+L)^{T}D^{-1}(D+L)$, $(D+L)D^{-1}(D+L)^{T}$, $(D+U)^{T}D^{-1}(D+U)$, $(D+U)D^{-1}(D+U)^{T}$, and their weighted variants are all SPD.
\end{remark}

To describe two-grid methods, we make the following basic assumptions.

\begin{itemize}[leftmargin=0.8cm]

\item Let $M\in\mathbb{R}^{n\times n}$ be an SPD smoother such that $\|I-M^{-1}A\|_{M}\leq 1$.

\item Let $R\in\mathbb{R}^{n_{\rm c}\times n}$ be a restriction matrix of full row rank, where $n_{\rm c} \ (<n)$ denotes the number of coarse variables.

\item Let $P\in\mathbb{R}^{n\times n_{\rm c}}$ be a prolongation (or interpolation) matrix of rank $n_{\rm c}$.

\item Assume that the coarse-grid matrix $A_{\rm c}:=RAP$ is nonsingular.

\end{itemize}
Under the above assumptions, a two-grid method for solving~\eqref{system} is described by Algorithm~\ref{alg:TG}, in which the coarse-grid system $A_{\rm c}\mathbf{e}_{\rm c}=\mathbf{r}_{\rm c}$ is solved exactly.
\begin{algorithm}[!htbp]

\caption{\ Exact two-grid method.}\label{alg:TG}

\smallskip

\begin{algorithmic}[1]

\State Smoothing: $\mathbf{u}^{(1)}\gets\mathbf{u}^{(0)}+M^{-1}\big(\mathbf{f}-A\mathbf{u}^{(0)}\big)$ \Comment{$\mathbf{u}^{(0)}\in\mathbb{R}^{n}$ is an initial guess}

\smallskip

\State Restriction: $\mathbf{r}_{\rm c}\gets R\big(\mathbf{f}-A\mathbf{u}^{(1)}\big)$

\smallskip

\State Coarse-grid correction: $\mathbf{e}_{\rm c}\gets A_{\rm c}^{-1}\mathbf{r}_{\rm c}$

\smallskip

\State Prolongation: $\mathbf{u}_{\rm\s TG}\gets\mathbf{u}^{(1)}+P\mathbf{e}_{\rm c}$

\smallskip

\end{algorithmic}

\end{algorithm}

Algorithm~\ref{alg:TG} involves two error propagation processes, the first of which is given by
\begin{equation}\label{error-1}
\mathbf{u}-\mathbf{u}^{(1)}=(I-M^{-1}A)\big(\mathbf{u}-\mathbf{u}^{(0)}\big).
\end{equation}
For any SPD matrix $X\in\mathbb{R}^{n\times n}$, we have
\begin{displaymath}
\big\|\mathbf{u}-\mathbf{u}^{(1)}\big\|_{X}\leq\|I-M^{-1}A\|_{X}\big\|\mathbf{u}-\mathbf{u}^{(0)}\big\|_{X}.
\end{displaymath}
If
\begin{displaymath}
\|I-M^{-1}A\|_{X}>1,
\end{displaymath}
then it is possible that
\begin{displaymath}
\big\|\mathbf{u}-\mathbf{u}^{(1)}\big\|_{X}>\big\|\mathbf{u}-\mathbf{u}^{(0)}\big\|_{X}.
\end{displaymath}
To avoid this, we assume that
\begin{displaymath}
\|I-M^{-1}A\|_{X}\leq 1,
\end{displaymath}
which is equivalent to
\begin{equation}\label{smooth-req}
\lambda_{\min}\big(X^{\frac{1}{2}}M^{-1}AX^{-\frac{1}{2}}+X^{-\frac{1}{2}}A^{T}M^{-1}X^{\frac{1}{2}}-X^{\frac{1}{2}}M^{-1}AX^{-1}A^{T}M^{-1}X^{\frac{1}{2}}\big)\geq 0.
\end{equation}
To make the SPD matrix $A+A^{T}$ appear explicitly in~\eqref{smooth-req}, we set
\begin{displaymath}
X^{\frac{1}{2}}M^{-1}=X^{-\frac{1}{2}},
\end{displaymath}
i.e., $X=M$. With this choice, \eqref{smooth-req} reduces to
\begin{displaymath}
\lambda_{\min}\big(M^{-\frac{1}{2}}\widetilde{A}M^{-\frac{1}{2}}\big)\geq 0,
\end{displaymath}
where
\begin{equation}\label{tildA}
\widetilde{A}:=A+A^{T}-AM^{-1}A^{T}.
\end{equation}
Thus, $\|I-M^{-1}A\|_{M}\leq 1$ if and only if $\widetilde{A}$ is positive semidefinite, or, equivalently, $M-A^{T}(A+A^{T})^{-1}A$ is positive semidefinite.

The second error propagation process of Algorithm~\ref{alg:TG} is given by
\begin{equation}\label{error-2}
\mathbf{u}-\mathbf{u}_{\rm\s TG}=(I-\Pi)\big(\mathbf{u}-\mathbf{u}^{(1)}\big),
\end{equation}
where $\Pi$ is defined by~\eqref{pi}. It follows from~\eqref{error-2} that
\begin{displaymath}
\|\mathbf{u}-\mathbf{u}_{\rm\s TG}\|_{M}\leq\|I-\Pi\|_{M}\big\|\mathbf{u}-\mathbf{u}^{(1)}\big\|_{M}.
\end{displaymath}
Since $I-\Pi$ is a nonzero projector, we have
\begin{displaymath}
\|I-\Pi\|_{M}\geq 1.
\end{displaymath}
Thus, a sufficient condition for $\|\mathbf{u}-\mathbf{u}_{\rm\s TG}\|_{M}\leq\big\|\mathbf{u}-\mathbf{u}^{(1)}\big\|_{M}$ is
\begin{displaymath}
\|I-\Pi\|_{M}=1,
\end{displaymath}
which holds if and only if
\begin{equation}\label{cond-P}
\mathcal{R}(P)=\mathcal{R}(M^{-1}A^{T}R^{T});
\end{equation}
see Remark~\ref{rmk:proj} for details. If $R$ and $P$ do not satisfy~\eqref{cond-P}, then
\begin{displaymath}
\|I-\Pi\|_{M}>1,
\end{displaymath}
which may lead to
\begin{displaymath}
\|\mathbf{u}-\mathbf{u}_{\rm\s TG}\|_{M}>\big\|\mathbf{u}-\mathbf{u}^{(1)}\big\|_{M}.
\end{displaymath}
This defeats the purpose of the coarse-grid correction.

\begin{remark}\label{rmk:proj}
For any norm $\|\cdot\|_{X}$, it holds that
\begin{displaymath}
\|I-\Pi\|_{X}\geq 1,
\end{displaymath}
with equality if and only if
\begin{equation}\label{cond-pi}
\Pi=X^{-1}\Pi^{T}X.
\end{equation}
Note that the left-hand side of~\eqref{cond-pi} is a projector along $\mathcal{N}(RA)$ onto $\mathcal{R}(P)$, whereas the right-hand side is a projector along $\mathcal{N}(P^{T}X)$ onto $\mathcal{R}(X^{-1}A^{T}R^{T})$. Hence, \eqref{cond-pi} implies
\begin{equation}\label{cond0-P}
\mathcal{R}(P)=\mathcal{R}(X^{-1}A^{T}R^{T}).
\end{equation}
Conversely, if~\eqref{cond0-P} holds, then
\begin{displaymath}
XP=A^{T}R^{T}W
\end{displaymath}
for some nonsingular matrix $W\in\mathbb{R}^{n_{\rm c}\times n_{\rm c}}$. Thus,
\begin{displaymath}
\mathcal{N}(RA)=\mathcal{N}(W^{T}RA)=\mathcal{N}(P^{T}X).
\end{displaymath}
This, together with the fact that both $\Pi$ and $X^{-1}\Pi^{T}X$ are projectors, yields~\eqref{cond-pi}. Therefore, for any norm $\|\cdot\|_{X}$, $\|I-\Pi\|_{X}=1$ is equivalent to~\eqref{cond0-P}.
\end{remark}

\begin{remark}
If $A$ is SPD, a natural choice is $X=A$. In this case, \eqref{cond0-P} reduces to
\begin{displaymath}
\mathcal{R}(P)=\mathcal{R}(R^{T}),
\end{displaymath}
which justifies the standard choice of prolongation matrices in both geometric and algebraic multigrid methods, namely, $P=\alpha R^{T}$ for some nonzero scalar $\alpha$.
\end{remark}

Combining~\eqref{error-1} and~\eqref{error-2}, we obtain
\begin{displaymath}
\mathbf{u}-\mathbf{u}_{\rm\s TG}=E_{\rm\s TG}\big(\mathbf{u}-\mathbf{u}^{(0)}\big),
\end{displaymath}
where $E_{\rm\s TG}$, called the \textit{iteration matrix} (or \textit{error propagation matrix}) of Algorithm~\ref{alg:TG}, is given by
\begin{equation}\label{ETG}
E_{\rm\s TG}=(I-\Pi)(I-M^{-1}A).
\end{equation}
It then holds that
\begin{displaymath}
\|\mathbf{u}-\mathbf{u}_{\rm\s TG}\|_{M}\leq\|E_{\rm\s TG}\|_{M}\big\|\mathbf{u}-\mathbf{u}^{(0)}\big\|_{M},
\end{displaymath}
where
\begin{equation}\label{ETG-norm}
\|E_{\rm\s TG}\|_{M}=\big\|\big(I-M^{\frac{1}{2}}\Pi M^{-\frac{1}{2}}\big)\big(I-M^{-\frac{1}{2}}AM^{-\frac{1}{2}}\big)\big\|_{2}.
\end{equation}
The quantity $\|E_{\rm\s TG}\|_{M}$ is referred to as the \textit{$M$-convergence factor} of Algorithm~\ref{alg:TG}.

\section{Convergence analysis of Algorithm~\ref{alg:TG}} \label{sec:TG}

Motivated by~\eqref{cond-P}, we define
\begin{displaymath}
P_{\s M}:=M^{-1}A^{T}R^{T},
\end{displaymath}
which can be viewed as a \textit{left multiplicative perturbation} of the classical choice $R^{T}$. It is easy to see that $P_{\s M}$ satisfies the relative error estimate
\begin{displaymath}
\frac{\|R^{T}-P_{\s M}\|_{M}}{\|R^{T}\|_{M}}\leq\|I-M^{-1}A^{T}\|_{M}=\|I-M^{-1}A\|_{M}\leq 1.
\end{displaymath}

In this section, we study the convergence of Algorithm~\ref{alg:TG} with prolongation matrix $P_{\s M}$. Specifically, we first present an identity for $\|E_{\rm\s TG}\|_{M}$, and then discuss its use in deriving a class of optimal restriction matrices and in analyzing the influence of $\mathcal{R}(R^{T})$ on $\|E_{\rm\s TG}\|_{M}$.

We first prove a technical lemma, which provides a necessary and sufficient condition for $\|E_{\rm\s TG}\|_{M}<1$; see the proof of Theorem~\ref{thm:identity} for details.

\begin{lemma}\label{lem:null}
Let $\Pi$ and $\widetilde{A}$ be defined by~\eqref{pi} and~\eqref{tildA}, respectively. Then
\begin{equation}\label{cond-tA}
\mathcal{N}(\widetilde{A})\cap\mathcal{N}(RA)=\{0\}
\end{equation}
if and only if
\begin{equation}\label{rank-n-nc}
\rank\big(\widetilde{A}^{\frac{1}{2}}(I-\Pi)\big)=n-n_{\rm c}.
\end{equation}
\end{lemma}

\begin{proof}
We first observe that~\eqref{rank-n-nc} is equivalent to
\begin{equation}\label{equal-null}
\mathcal{N}\big(\widetilde{A}^{\frac{1}{2}}(I-\Pi)\big)=\mathcal{N}(I-\Pi).
\end{equation}
Thus, it suffices to show that~\eqref{cond-tA} is equivalent to~\eqref{equal-null}.

``\eqref{cond-tA}\,$\Rightarrow$\,\eqref{equal-null}'': For any $\mathbf{x}\in\mathcal{N}\big(\widetilde{A}^{\frac{1}{2}}(I-\Pi)\big)$, we have
\begin{displaymath}
\widetilde{A}(I-\Pi)\mathbf{x}=0,
\end{displaymath}
which implies $(I-\Pi)\mathbf{x}\in\mathcal{N}(\widetilde{A})$. Note that 
\begin{displaymath}
(I-\Pi)\mathbf{x}\in\mathcal{N}(\Pi)=\mathcal{N}(RA).
\end{displaymath}
If~\eqref{cond-tA} holds, then $\mathbf{x}\in\mathcal{N}(I-\Pi)$. The arbitrariness of $\mathbf{x}$ implies
\begin{displaymath}
\mathcal{N}\big(\widetilde{A}^{\frac{1}{2}}(I-\Pi)\big)\subseteq\mathcal{N}(I-\Pi),
\end{displaymath}
which, combined with the reverse inclusion $\mathcal{N}(I-\Pi)\subseteq\mathcal{N}\big(\widetilde{A}^{\frac{1}{2}}(I-\Pi)\big)$, yields~\eqref{equal-null}.

``\eqref{equal-null}\,$\Rightarrow$\,\eqref{cond-tA}'': Assume that there exists a nonzero vector $\mathbf{y}\in\mathcal{N}(\widetilde{A})\cap\mathcal{N}(RA)$, that is,~\eqref{cond-tA} does not hold. Since $\mathbf{y}\neq 0$ and
\begin{displaymath}
\mathcal{N}(RA)=\mathcal{R}(I-\Pi),
\end{displaymath}
we can write $\mathbf{y}=(I-\Pi)\mathbf{z}$ for some $\mathbf{z}\in\mathbb{R}^{n}\backslash\mathcal{R}(P)$, which, together with $\mathbf{y}\in\mathcal{N}(\widetilde{A})$ and $\mathcal{N}(\widetilde{A})=\mathcal{N}\big(\widetilde{A}^{\frac{1}{2}}\big)$, yields
\begin{displaymath}
\widetilde{A}^{\frac{1}{2}}(I-\Pi)\mathbf{z}=0.
\end{displaymath}
Hence, $\mathbf{z}\in\mathcal{N}\big(\widetilde{A}^{\frac{1}{2}}(I-\Pi)\big)$ but $\mathbf{z}\notin\mathcal{N}(I-\Pi)$. This contradicts~\eqref{equal-null}.
\end{proof}

\begin{remark}
Clearly, the relation~\eqref{cond-tA} holds if $\widetilde{A}$ is positive definite, or, equivalently, $\|I-M^{-1}A\|_{M}<1$. In addition, \eqref{cond-tA} implies
\begin{displaymath}
\rank(\widetilde{A})\geq n-n_{\rm c},
\end{displaymath}
since
\begin{displaymath}
\dim\big(\mathcal{N}(\widetilde{A})\big)+\dim\big(\mathcal{N}(RA)\big)=\dim\big(\mathcal{N}(\widetilde{A})+\mathcal{N}(RA)\big)\leq n.
\end{displaymath}
Here, $\dim(\cdot)$ denotes the dimension of a subspace of $\mathbb{R}^{n}$. \textit{Throughout this paper, we assume that~\eqref{cond-tA} holds}.
\end{remark}

Using~\eqref{ETG-norm} and Lemma~\ref{lem:null}, we derive the following convergence identity.

\begin{theorem}\label{thm:identity}
The $M$-convergence factor of Algorithm~{\rm\ref{alg:TG}} with prolongation matrix $P_{\s M}$ can be characterized as
\begin{equation}\label{identity}
\|E_{\rm\s TG}\|_{M}=\sqrt{1-\sigma_{\rm\s TG}}
\end{equation}
with
\begin{displaymath}
\sigma_{\rm\s TG}=\lambda_{n_{\rm c}+1}\big(M^{-1}\widetilde{A}(I-\Pi)\big),
\end{displaymath}
where $\widetilde{A}$ and $\Pi$ are defined by~\eqref{tildA} and~\eqref{pi}, respectively.
\end{theorem}

\begin{proof}
Let
\begin{displaymath}
\Pi_{\s M}=M^{\frac{1}{2}}\Pi M^{-\frac{1}{2}}.
\end{displaymath}
If $P=P_{\s M}$, then
\begin{displaymath}
\Pi=P_{\s M}A_{\rm c}^{-1}RA=P_{\s M}(P_{\s M}^{T}MP_{\s M})^{-1}P_{\s M}^{T}M,
\end{displaymath}
and hence
\begin{displaymath}
\Pi_{\s M}=M^{\frac{1}{2}}P_{\s M}(P_{\s M}^{T}MP_{\s M})^{-1}P_{\s M}^{T}M^{\frac{1}{2}}.
\end{displaymath}
It can be readily verified that
\begin{displaymath}
\Pi_{\s M}^{T}=\Pi_{\s M}=\Pi_{\s M}^{2}.
\end{displaymath}
By~\eqref{ETG-norm}, we have
\begin{align*}
\|E_{\rm\s TG}\|_{M}^{2}&=\big\|(I-\Pi_{\s M})\big(I-M^{-\frac{1}{2}}AM^{-\frac{1}{2}}\big)\big\|_{2}^{2}\\
&=\lambda_{\max}\big(\big(I-M^{-\frac{1}{2}}A^{T}M^{-\frac{1}{2}}\big)(I-\Pi_{\s M})\big(I-M^{-\frac{1}{2}}AM^{-\frac{1}{2}}\big)\big)\\
&=\lambda_{\max}\big(\big(I-M^{-\frac{1}{2}}AM^{-\frac{1}{2}}\big)\big(I-M^{-\frac{1}{2}}A^{T}M^{-\frac{1}{2}}\big)(I-\Pi_{\s M})\big)\\
&=\lambda_{\max}\big(\big(I-M^{-\frac{1}{2}}\widetilde{A}M^{-\frac{1}{2}}\big)(I-\Pi_{\s M})\big)\\
&=1-\lambda_{\min}\big(\Pi_{\s M}+M^{-\frac{1}{2}}\widetilde{A}M^{-\frac{1}{2}}(I-\Pi_{\s M})\big).
\end{align*}

Since $\Pi_{\s M}$ is an $L^{2}$-orthogonal projector of rank $n_{\rm c}$, there is an orthogonal matrix $Q\in\mathbb{R}^{n\times n}$ such that
\begin{displaymath}
Q^{T}\Pi_{\s M}Q=\begin{pmatrix}
I_{n_{\rm c}} & 0 \\
0 & 0
\end{pmatrix}.
\end{displaymath}
Let
\begin{displaymath}
Q^{T}M^{-\frac{1}{2}}\widetilde{A}M^{-\frac{1}{2}}Q=\begin{pmatrix}
X_{1} & X_{2} \\
X_{2}^{T} & X_{3}
\end{pmatrix},
\end{displaymath}
where $X_{1}\in\mathbb{R}^{n_{\rm c}\times n_{\rm c}}$, $X_{2}\in\mathbb{R}^{n_{\rm c}\times(n-n_{\rm c})}$, and $X_{3}\in\mathbb{R}^{(n-n_{\rm c})\times(n-n_{\rm c})}$. Then
\begin{align*}
\begin{pmatrix}
I_{n_{\rm c}}-X_{1} & -X_{2} \\
-X_{2}^{T} & I_{n-n_{\rm c}}-X_{3}
\end{pmatrix}&=Q^{T}\big(I-M^{-\frac{1}{2}}\widetilde{A}M^{-\frac{1}{2}}\big)Q \\
&=Q^{T}\big(I-M^{-\frac{1}{2}}AM^{-\frac{1}{2}}\big)\big(I-M^{-\frac{1}{2}}A^{T}M^{-\frac{1}{2}}\big)Q,
\end{align*}
which, together with the positive semidefiniteness of $Q^{T}M^{-\frac{1}{2}}\widetilde{A}M^{-\frac{1}{2}}Q$, yields
\begin{equation}\label{spe-X3}
\lambda(X_{3})\subset[0,1].
\end{equation}
Direct computation gives
\begin{displaymath}
\Pi_{\s M}+M^{-\frac{1}{2}}\widetilde{A}M^{-\frac{1}{2}}(I-\Pi_{\s M})=Q\begin{pmatrix}
I_{n_{\rm c}} & X_{2} \\
0 & X_{3}
\end{pmatrix}Q^{T}.
\end{displaymath}
Then
\begin{align*}
\|E_{\rm\s TG}\|_{M}^{2}&=1-\lambda_{\min}\big(\Pi_{\s M}+M^{-\frac{1}{2}}\widetilde{A}M^{-\frac{1}{2}}(I-\Pi_{\s M})\big)\\
&=1-\min\big\{1,\,\lambda_{\min}(X_{3})\big\}\\
&=1-\lambda_{\min}(X_{3}).
\end{align*}

Using Lemma~\ref{lem:null} and the relation
\begin{displaymath}
(I-\Pi_{\s M})M^{-\frac{1}{2}}\widetilde{A}M^{-\frac{1}{2}}(I-\Pi_{\s M})=Q\begin{pmatrix}
0 & 0 \\
0 & X_{3}
\end{pmatrix}Q^{T},
\end{displaymath}
we obtain
\begin{align*}
\rank(X_{3})&=\rank\big((I-\Pi_{\s M})M^{-\frac{1}{2}}\widetilde{A}M^{-\frac{1}{2}}(I-\Pi_{\s M})\big)\\
&=\rank\big(\widetilde{A}^{\frac{1}{2}}M^{-\frac{1}{2}}(I-\Pi_{\s M})\big)\\
&=\rank\big(\widetilde{A}^{\frac{1}{2}}(I-\Pi)\big)\\
&=n-n_{\rm c},
\end{align*}
which, combined with~\eqref{spe-X3}, leads to $\lambda_{\min}(X_{3})>0$. Due to
\begin{equation}\label{MtAI-pi}
M^{-\frac{1}{2}}\widetilde{A}M^{-\frac{1}{2}}(I-\Pi_{\s M})=Q\begin{pmatrix}
0 & X_{2} \\
0 & X_{3}
\end{pmatrix}Q^{T},
\end{equation}
it follows that
\begin{displaymath}
\lambda_{\min}(X_{3})=\lambda_{n_{\rm c}+1}\big(M^{-\frac{1}{2}}\widetilde{A}M^{-\frac{1}{2}}(I-\Pi_{\s M})\big).
\end{displaymath}
Thus,
\begin{align*}
\|E_{\rm\s TG}\|_{M}^{2}&=1-\lambda_{n_{\rm c}+1}\big(M^{-\frac{1}{2}}\widetilde{A}M^{-\frac{1}{2}}(I-\Pi_{\s M})\big)\\
&=1-\lambda_{n_{\rm c}+1}\big(M^{-1}\widetilde{A}M^{-\frac{1}{2}}(I-\Pi_{\s M})M^{\frac{1}{2}}\big)\\
&=1-\lambda_{n_{\rm c}+1}\big(M^{-1}\widetilde{A}(I-\Pi)\big),
\end{align*}
which yields~\eqref{identity}.
\end{proof}

The identity~\eqref{identity} serves as a convenient tool for analyzing Algorithm~\ref{alg:TG}. Next, we present two applications of~\eqref{identity}. The first is to derive a class of optimal restriction matrices that minimize the convergence factor $\|E_{\rm\s TG}\|_{M}$.

To analyze the optimality of the restriction matrix, we need the following result, known as the \textit{Poincar\'{e} separation theorem}; see, e.g.,~\cite[Corollary~4.3.37]{Horn2013}.

\begin{lemma}\label{lem:poincare}
Let $H\in\mathbb{C}^{n\times n}$ be Hermitian, and let $\{\mathbf{q}_{k}\}_{k=1}^{m}\subset\mathbb{C}^{n}\,(1\leq m\leq n)$ be a set of orthonormal vectors. Then, for each $i=1,\ldots,m$, it holds that
\begin{displaymath}
\lambda_{i}(H)\leq\lambda_{i}(\check{H})\leq\lambda_{i+n-m}(H),
\end{displaymath}
where $\check{H}=\big(\mathbf{q}_{i}^{\ast}H\mathbf{q}_{j}\big)\in\mathbb{C}^{m\times m}$ and $\mathbf{q}_{i}^{\ast}$ denotes the conjugate transpose of $\mathbf{q}_{i}$.
\end{lemma}

The following theorem provides an optimal restriction theory.

\begin{theorem}
Let $\widetilde{A}$ be defined by~\eqref{tildA}, and let $\{(\mu_{i},\mathbf{v}_{i})\}_{i=1}^{n}$ be the eigenpairs of the generalized eigenvalue problem
\begin{displaymath}
\widetilde{A}\mathbf{v}=\mu M\mathbf{v},
\end{displaymath}
where
\begin{displaymath}
0\leq\mu_{1}\leq\mu_{2}\leq\cdots\leq\mu_{n}\leq 1 \quad \text{and} \quad \mathbf{v}_{i}^{T}M\mathbf{v}_{j}=\begin{cases}
1 & \text{if $i=j$},\\
0 & \text{if $i\neq j$}.
\end{cases}
\end{displaymath}
Then
\begin{displaymath}
\|E_{\rm\s TG}\|_{M}\geq\sqrt{1-\mu_{n_{\rm c}+1}},
\end{displaymath}
with equality if $\mathcal{N}(R)=\Span\{A\mathbf{v}_{n_{\rm c}+1},\ldots,A\mathbf{v}_{n}\}$.
\end{theorem}

\begin{proof}
Since
\begin{displaymath}
\lambda\big(M^{-1}\widetilde{A}\big)=\lambda\big(M^{-\frac{1}{2}}\widetilde{A}M^{-\frac{1}{2}}\big)\subset[0,+\infty)
\end{displaymath}
and $I-M^{-\frac{1}{2}}\widetilde{A}M^{-\frac{1}{2}}$ is symmetric positive semidefinite, it follows that
\begin{displaymath}
0=\mu_{1}=\cdots=\mu_{n-r}<\mu_{n-r+1}\leq\cdots\leq\mu_{n}\leq 1,
\end{displaymath}
where $r=\rank(\widetilde{A})\geq n-n_{\rm c}$ (which implies $\mu_{n_{\rm c}+1}>0$).

Let
\begin{displaymath}
V=(\mathbf{v}_{1},\ldots,\mathbf{v}_{n}) \quad \text{and} \quad U_{1}=V^{-1}P_{\s M}\big(P_{\s M}^{T}V^{-T}V^{-1}P_{\s M}\big)^{-\frac{1}{2}}.
\end{displaymath}
It is straightforward to verify that $V$ is nonsingular with $V^{-1}=V^{T}M$, and that $U_{1}$ has orthonormal columns, i.e., $U_{1}^{T}U_{1}=I_{n_{\rm c}}$. Let $U_{2}\in\mathbb{R}^{n\times(n-n_{\rm c})}$ be such that $(U_{1} \,\ U_{2})$ is orthogonal. Then
\begin{align*}
M^{-1}\widetilde{A}(I-\Pi)&=M^{-1}\widetilde{A}\big(I-P_{\s M}(P_{\s M}^{T}MP_{\s M})^{-1}P_{\s M}^{T}M\big)\\
&=M^{-1}\widetilde{A}\big(I-VU_{1}(U_{1}^{T}V^{T}MVU_{1})^{-1}U_{1}^{T}V^{T}M\big)\\
&=M^{-1}\widetilde{A}(I-VU_{1}U_{1}^{T}V^{T}M)\\
&=M^{-1}\widetilde{A}(I-VU_{1}U_{1}^{T}V^{-1})\\
&=M^{-1}\widetilde{A}VU_{2}U_{2}^{T}V^{-1}\\
&=V\Lambda U_{2}U_{2}^{T}V^{-1},
\end{align*}
where $\Lambda=\diag\big(0,\ldots,0,\mu_{n-r+1},\ldots,\mu_{n}\big)\in\mathbb{R}^{n\times n}$. Hence,
\begin{align*}
\lambda\big(\Lambda U_{2}U_{2}^{T}\big)&=\lambda\big(M^{-1}\widetilde{A}(I-\Pi)\big)\\
&=\lambda\big(M^{-\frac{1}{2}}\widetilde{A}M^{-\frac{1}{2}}M^{\frac{1}{2}}(I-\Pi)M^{-\frac{1}{2}}\big)\\
&=\lambda\big(M^{-\frac{1}{2}}\widetilde{A}M^{-\frac{1}{2}}(I-\Pi_{\s M})\big).
\end{align*}
This, together with~\eqref{MtAI-pi} and the positive definiteness of $X_{3}$, yields that $\Lambda U_{2}U_{2}^{T}$ has $n_{\rm c}$ zero eigenvalues and $n-n_{\rm c}$ positive eigenvalues. Applying~\eqref{identity} and Lemma~\ref{lem:poincare}, we obtain
\begin{align*}
\|E_{\rm\s TG}\|_{M}&=\sqrt{1-\lambda_{n_{\rm c}+1}\big(M^{-1}\widetilde{A}(I-\Pi)\big)}\\
&=\sqrt{1-\lambda_{n_{\rm c}+1}\big(\Lambda U_{2}U_{2}^{T}\big)}\\
&=\sqrt{1-\lambda_{1}\big(U_{2}^{T}\Lambda U_{2}\big)}\\
&\geq\sqrt{1-\mu_{n_{\rm c}+1}}.
\end{align*}

In particular, if $\mathcal{N}(R)=\Span\{A\mathbf{v}_{n_{\rm c}+1},\ldots,A\mathbf{v}_{n}\}$, then
\begin{displaymath}
V^{-1}P_{\s M}=V^{-1}M^{-1}A^{T}R^{T}=V^{T}A^{T}R^{T}=(RAV)^{T}=\begin{pmatrix}
R_{\rm c}^{T} \\
0
\end{pmatrix},
\end{displaymath}
where $R_{\rm c}=RA(\mathbf{v}_{1},\ldots,\mathbf{v}_{n_{\rm c}})\in\mathbb{R}^{n_{\rm c}\times n_{\rm c}}$ is nonsingular. We then have
\begin{align*}
U_{2}U_{2}^{T}&=I-U_{1}U_{1}^{T}\\
&=I-V^{-1}P_{\s M}\big(P_{\s M}^{T}V^{-T}V^{-1}P_{\s M}\big)^{-1}P_{\s M}^{T}V^{-T}\\
&=I-\begin{pmatrix}
R_{\rm c}^{T} \\
0
\end{pmatrix}R_{\rm c}^{-T}R_{\rm c}^{-1}\big(R_{\rm c} \,\ 0\big)\\
&=\begin{pmatrix}
0 & 0 \\
0 & I_{n-n_{\rm c}}
\end{pmatrix}.
\end{align*}
Thus,
\begin{displaymath}
\|E_{\rm\s TG}\|_{M}=\sqrt{1-\lambda_{n_{\rm c}+1}\big(\Lambda U_{2}U_{2}^{T}\big)}=\sqrt{1-\mu_{n_{\rm c}+1}}.
\end{displaymath}
This completes the proof.
\end{proof}

The second application of~\eqref{identity} is to analyze the influence of $\mathcal{R}(R^{T})$ on $\|E_{\rm\s TG}\|_{M}$. The following lemma is needed in our analysis; see, e.g.,~\cite[Corollary~4.3.5]{Horn2013}.

\begin{lemma}\label{lem:H1H2}
Let $H_{1},H_{2}\in\mathbb{C}^{n\times n}$ be Hermitian. If $H_{2}$ is singular, then
\begin{displaymath}
\lambda_{i}(H_{1}+H_{2})\leq\lambda_{i+\rank(H_{2})}(H_{1})
\end{displaymath}
for all $i=1,\ldots,n-\rank(H_{2})$.
\end{lemma}

The following theorem shows that the convergence factor $\|E_{\rm\s TG}\|_{M}$ is nonincreasing as $\mathcal{R}(R^{T})$ expands.

\begin{theorem}
Let $\widehat{R}\in\mathbb{R}^{\hat{n}_{\rm c}\times n}\,(n_{\rm c}\leq\hat{n}_{\rm c}<n)$ be of full row rank, and define
\begin{displaymath}
\widehat{P}_{\s M}:=M^{-1}A^{T}\widehat{R}^{T}.
\end{displaymath}
If
\begin{displaymath}
\mathcal{R}(R^{T})\subseteq\mathcal{R}(\widehat{R}^{T}),
\end{displaymath}
then
\begin{equation}\label{hatETG-ETG}
\|\widehat{E}_{\rm\s TG}\|_{M}\leq\|E_{\rm\s TG}\|_{M},
\end{equation}
where
\begin{displaymath}
\widehat{E}_{\rm\s TG}=\big(I-\widehat{P}_{\s M}(\widehat{R}A\widehat{P}_{\s M})^{-1}\widehat{R}A\big)(I-M^{-1}A)
\end{displaymath}
and $E_{\rm\s TG}$ is given by~\eqref{ETG}.
\end{theorem}

\begin{proof}
The condition $\mathcal{R}(R^{T})\subseteq\mathcal{R}(\widehat{R}^{T})$ implies that there exists $Y\in\mathbb{R}^{\hat{n}_{\rm c}\times n_{\rm c}}$ of full column rank such that
\begin{displaymath}
R^{T}=\widehat{R}^{T}Y.
\end{displaymath}
Let $\widehat{Y}\in\mathbb{R}^{\hat{n}_{\rm c}\times\hat{n}_{\rm c}}$ be nonsingular such that
\begin{displaymath}
Y=\widehat{Y}\begin{pmatrix}
I_{n_{\rm c}} \\
0
\end{pmatrix}.
\end{displaymath}
Then
\begin{displaymath}
R^{T}=\widehat{R}^{T}\widehat{Y}\begin{pmatrix}
I_{n_{\rm c}} \\
0
\end{pmatrix},
\end{displaymath}
and hence there exists $Z_{0}\in\mathbb{R}^{n\times(\hat{n}_{\rm c}-n_{\rm c})}$ of full column rank such that
\begin{equation}\label{Rhat-R}
\widehat{R}^{T}=\big(R^{T} \,\ Z_{0}\big)\widehat{Y}^{-1}.
\end{equation}

From the proof of Theorem~\ref{thm:identity}, we have
\begin{align*}
\sigma_{\rm\s TG}&=\lambda_{n_{\rm c}+1}\big(M^{-1}\widetilde{A}\big(I-P_{\s M}(P_{\s M}^{T}MP_{\s M})^{-1}P_{\s M}^{T}M\big)\big)\\
&=\lambda_{n_{\rm c}+1}\big(\widetilde{A}\big(I-P_{\s M}(P_{\s M}^{T}MP_{\s M})^{-1}P_{\s M}^{T}M\big)M^{-1}\big)\\
&=\lambda_{n_{\rm c}+1}\big(\widetilde{A}^{\frac{1}{2}}\big(M^{-1}-P_{\s M}(P_{\s M}^{T}MP_{\s M})^{-1}P_{\s M}^{T}\big)\widetilde{A}^{\frac{1}{2}}\big).
\end{align*}
Similarly,
\begin{displaymath}
\|\widehat{E}_{\rm\s TG}\|_{M}=\sqrt{1-\hat{\sigma}_{\rm\s TG}}
\end{displaymath}
with
\begin{displaymath}
\hat{\sigma}_{\rm\s TG}=\lambda_{\hat{n}_{\rm c}+1}\big(\widetilde{A}^{\frac{1}{2}}\big(M^{-1}-\widehat{P}_{\s M}(\widehat{P}_{\s M}^{T}M\widehat{P}_{\s M})^{-1}\widehat{P}_{\s M}^{T}\big)\widetilde{A}^{\frac{1}{2}}\big).
\end{displaymath}

By~\eqref{Rhat-R}, we have
\begin{displaymath}
\widehat{P}_{\s M}=\big(P_{\s M} \,\ Z\big)\widehat{Y}^{-1},
\end{displaymath}
where
\begin{displaymath}
Z=M^{-1}A^{T}Z_{0}.
\end{displaymath}
Let
\begin{align*}
D_{\s M}&=\widehat{P}_{\s M}(\widehat{P}_{\s M}^{T}M\widehat{P}_{\s M})^{-1}\widehat{P}_{\s M}^{T}-P_{\s M}(P_{\s M}^{T}MP_{\s M})^{-1}P_{\s M}^{T},\\
L_{\s M}&=\begin{pmatrix}
I_{n_{\rm c}} & 0 \\
-Z^{T}MP_{\s M}(P_{\s M}^{T}MP_{\s M})^{-1} & I_{\hat{n}_{\rm c}-n_{\rm c}}
\end{pmatrix},\\
S_{\s M}&=Z^{T}MZ-Z^{T}MP_{\s M}(P_{\s M}^{T}MP_{\s M})^{-1}P_{\s M}^{T}MZ.
\end{align*}
Then
\begin{align*}
D_{\s M}&=\big(P_{\s M} \,\ Z\big)\bigg[\begin{pmatrix}
P_{\s M}^{T}MP_{\s M} & P_{\s M}^{T}MZ \\
Z^{T}MP_{\s M} & Z^{T}MZ
\end{pmatrix}^{-1}-\begin{pmatrix}
(P_{\s M}^{T}MP_{\s M})^{-1} & 0 \\
0 & 0
\end{pmatrix}\bigg]\big(P_{\s M} \,\ Z\big)^{T}\\
&=\big(P_{\s M} \,\ Z\big)\bigg[L_{\s M}^{T}\begin{pmatrix}
(P_{\s M}^{T}MP_{\s M})^{-1} & 0 \\
0 & S_{\s M}^{-1}
\end{pmatrix}L_{\s M}-\begin{pmatrix}
(P_{\s M}^{T}MP_{\s M})^{-1} & 0 \\
0 & 0
\end{pmatrix}\bigg]\big(P_{\s M} \,\ Z\big)^{T}\\
&=\big(P_{\s M} \,\ Z\big)\begin{pmatrix}
-(P_{\s M}^{T}MP_{\s M})^{-1}P_{\s M}^{T}MZ \\
I_{\hat{n}_{\rm c}-n_{\rm c}}
\end{pmatrix}S_{\s M}^{-1}\begin{pmatrix}
-(P_{\s M}^{T}MP_{\s M})^{-1}P_{\s M}^{T}MZ \\
I_{\hat{n}_{\rm c}-n_{\rm c}}
\end{pmatrix}^{T}\big(P_{\s M} \,\ Z\big)^{T}.
\end{align*}
It follows that $\widetilde{A}^{\frac{1}{2}}D_{\s M}\widetilde{A}^{\frac{1}{2}}$ is symmetric positive semidefinite and
\begin{displaymath}
\rank\big(\widetilde{A}^{\frac{1}{2}}D_{\s M}\widetilde{A}^{\frac{1}{2}}\big)\leq\rank(D_{\s M})\leq\hat{n}_{\rm c}-n_{\rm c}.
\end{displaymath}
Applying Lemma~\ref{lem:H1H2}, we obtain
\begin{align*}
\sigma_{\rm\s TG}&=\lambda_{n_{\rm c}+1}\big(\widetilde{A}^{\frac{1}{2}}\big(M^{-1}-P_{\s M}(P_{\s M}^{T}MP_{\s M})^{-1}P_{\s M}^{T}\big)\widetilde{A}^{\frac{1}{2}}\big)\\
&=\lambda_{n_{\rm c}+1}\big(\widetilde{A}^{\frac{1}{2}}\big(M^{-1}-\widehat{P}_{\s M}(\widehat{P}_{\s M}^{T}M\widehat{P}_{\s M})^{-1}\widehat{P}_{\s M}^{T}\big)\widetilde{A}^{\frac{1}{2}}+\widetilde{A}^{\frac{1}{2}}D_{\s M}\widetilde{A}^{\frac{1}{2}}\big)\\
&\leq\lambda_{n_{\rm c}+1+\hat{n}_{\rm c}-n_{\rm c}}\big(\widetilde{A}^{\frac{1}{2}}\big(M^{-1}-\widehat{P}_{\s M}(\widehat{P}_{\s M}^{T}M\widehat{P}_{\s M})^{-1}\widehat{P}_{\s M}^{T}\big)\widetilde{A}^{\frac{1}{2}}\big)\\
&=\lambda_{\hat{n}_{\rm c}+1}\big(\widetilde{A}^{\frac{1}{2}}\big(M^{-1}-\widehat{P}_{\s M}(\widehat{P}_{\s M}^{T}M\widehat{P}_{\s M})^{-1}\widehat{P}_{\s M}^{T}\big)\widetilde{A}^{\frac{1}{2}}\big)\\
&=\hat{\sigma}_{\rm\s TG}.
\end{align*}
Thus,
\begin{displaymath}
\|E_{\rm\s TG}\|_{M}=\sqrt{1-\sigma_{\rm\s TG}}\geq\sqrt{1-\hat{\sigma}_{\rm\s TG}}=\|\widehat{E}_{\rm\s TG}\|_{M},
\end{displaymath}
which proves~\eqref{hatETG-ETG}.
\end{proof}

\section{Analysis of an inexact variant of Algorithm~\ref{alg:TG}} \label{sec:iTG}

In Algorithm~\ref{alg:TG}, the coarse-grid system to be solved is
\begin{equation}\label{coarse-pro}
A_{\rm c}\mathbf{e}_{\rm c}=\mathbf{r}_{\rm c}.
\end{equation}
In this subsection, we consider an inexact variant of Algorithm~\ref{alg:TG}, in which~\eqref{coarse-pro} is solved approximately. This variant is presented in Algorithm~\ref{alg:iTG} and consists of the following components:
\begin{itemize}[leftmargin=0.8cm]

\item $\mathbf{u}^{(0)}\in\mathbb{R}^{n}$ is an initial guess;

\item $\nu$ denotes the number of smoothing steps;

\item $M_{k}\in\mathbb{R}^{n\times n}\,(k=1,2,\ldots,\nu)$ are nonsingular smoothers;

\item $R\in\mathbb{R}^{n_{\rm c}\times n}$ is of full row rank;

\item $\mathscr{B}_{\rm c}\llbracket\cdot\rrbracket$ is a general mapping from $\mathbb{R}^{n_{\rm c}}$ to $\mathbb{R}^{n_{\rm c}}$;

\item $P_{\s X}:=X^{-1}A^{T}R^{T}$, where $X\in\mathbb{R}^{n\times n}$ is a generic SPD matrix.

\end{itemize}

\begin{algorithm}[!htbp]

\caption{\ Inexact two-grid method.}\label{alg:iTG}

\smallskip

\begin{algorithmic}[1]

\State $\nu$-Smoothing: $\mathbf{u}^{(k)}\gets\mathbf{u}^{(k-1)}+M_{k}^{-1}\big(\mathbf{f}-A\mathbf{u}^{(k-1)}\big)$ for $k=1,2,\ldots,\nu$
	
\smallskip

\State Restriction: $\mathbf{r}_{\rm c}^{(\nu)}\gets R\big(\mathbf{f}-A\mathbf{u}^{(\nu)}\big)$

\smallskip

\State Coarse-grid correction: $\widetilde{\mathbf{e}}_{\rm c}^{(\nu)}\gets\mathscr{B}_{\rm c}\big\llbracket\mathbf{r}_{\rm c}^{(\nu)}\big\rrbracket$

\smallskip

\State Prolongation: $\mathbf{u}_{\rm\s ITG}^{(\nu)}\gets\mathbf{u}^{(\nu)}+P_{\s X}\widetilde{\mathbf{e}}_{\rm c}^{(\nu)}$

\smallskip

\end{algorithmic}

\end{algorithm}

In particular, if $\mathscr{B}_{\rm c}\llbracket\cdot\rrbracket=A_{\rm c}^{-1}(\cdot)$ with
\begin{displaymath}
A_{\rm c}=RAP_{\s X}=P_{\s X}^{T}XP_{\s X},
\end{displaymath}
then the outputs of the correction and prolongation steps in Algorithm~\ref{alg:iTG} are denoted by $\mathbf{e}_{\rm c}^{(\nu)}$ and $\mathbf{u}_{\rm\s TG}^{(\nu)}$, respectively. That is,
\begin{equation}\label{Ae=r}
A_{\rm c}\mathbf{e}_{\rm c}^{(\nu)}=\mathbf{r}_{\rm c}^{(\nu)}
\end{equation}
and
\begin{equation}\label{utg}
\mathbf{u}_{\rm\s TG}^{(\nu)}=\mathbf{u}^{(\nu)}+P_{\s X}\mathbf{e}_{\rm c}^{(\nu)}.
\end{equation}

The error propagation of the smoothing process in Algorithm~\ref{alg:iTG} is given by
\begin{displaymath}
\mathbf{u}-\mathbf{u}^{(\nu)}=\big(I-M_{\nu}^{-1}A\big)\big(I-M_{\nu-1}^{-1}A\big)\cdots\big(I-M_{1}^{-1}A\big)\big(\mathbf{u}-\mathbf{u}^{(0)}\big).
\end{displaymath}
For any norm $\|\cdot\|_{X}$, we have
\begin{displaymath}
\big\|\mathbf{u}-\mathbf{u}^{(\nu)}\big\|_{X}\leq\big\|\big(I-M_{\nu}^{-1}A\big)\big(I-M_{\nu-1}^{-1}A\big)\cdots\big(I-M_{1}^{-1}A\big)\big\|_{X}\big\|\mathbf{u}-\mathbf{u}^{(0)}\big\|_{X}.
\end{displaymath}
An essential requirement for the smoothing process is
\begin{displaymath}
\big\|\mathbf{u}-\mathbf{u}^{(\nu)}\big\|_{X}\leq\big\|\mathbf{u}-\mathbf{u}^{(0)}\big\|_{X},
\end{displaymath}
that is, the smoothing process does not amplify the error, measured in the $X$-norm. By~\eqref{Ae=r} and~\eqref{utg}, we have
\begin{align*}
\mathbf{u}-\mathbf{u}_{\rm\s TG}^{(\nu)}&=\mathbf{u}-\mathbf{u}^{(\nu)}-P_{\s X}A_{\rm c}^{-1}\mathbf{r}_{\rm c}^{(\nu)}\\
&=\mathbf{u}-\mathbf{u}^{(\nu)}-P_{\s X}A_{\rm c}^{-1}RA\big(\mathbf{u}-\mathbf{u}^{(\nu)}\big)\\
&=\big(I-P_{\s X}A_{\rm c}^{-1}P_{\s X}^{T}X\big)\big(\mathbf{u}-\mathbf{u}^{(\nu)}\big).
\end{align*}
Then
\begin{displaymath}
\big\|\mathbf{u}-\mathbf{u}_{\rm\s TG}^{(\nu)}\big\|_{X}\leq\big\|I-P_{\s X}A_{\rm c}^{-1}P_{\s X}^{T}X\big\|_{X}\big\|\mathbf{u}-\mathbf{u}^{(\nu)}\big\|_{X}.
\end{displaymath}
Since $I-P_{\s X}A_{\rm c}^{-1}P_{\s X}^{T}X$ is an orthogonal projector with respect to the $X$-inner product, it follows that
\begin{displaymath}
\big\|I-P_{\s X}A_{\rm c}^{-1}P_{\s X}^{T}X\big\|_{X}=1,
\end{displaymath}
and hence
\begin{displaymath}
\big\|\mathbf{u}-\mathbf{u}_{\rm\s TG}^{(\nu)}\big\|_{X}\leq\big\|\mathbf{u}-\mathbf{u}^{(\nu)}\big\|_{X},
\end{displaymath}
that is, the correction process does not amplify the error. These discussions justify the assumptions in the following theorem, which provides a convergence estimate for Algorithm~\ref{alg:iTG}.

\begin{theorem}\label{thm:iTG}
Let $\tau_{\nu}\in(0,1]$ and $\rho_{\nu}\in(0,1)$ be parameters. Assume that
\begin{equation}\label{smoothing}
\big\|\mathbf{u}-\mathbf{u}^{(\nu)}\big\|_{X}\leq\tau_{\nu}\big\|\mathbf{u}-\mathbf{u}^{(0)}\big\|_{X},
\end{equation}
and that Algorithm~{\rm\ref{alg:iTG}} with $\mathscr{B}_{\rm c}\llbracket\cdot\rrbracket=A_{\rm c}^{-1}(\cdot)$ is convergent in the $X$-norm, i.e.,
\begin{equation}\label{TG-est}
\big\|\mathbf{u}-\mathbf{u}_{\rm\s TG}^{(\nu)}\big\|_{X}\leq\rho_{\nu}\big\|\mathbf{u}-\mathbf{u}^{(0)}\big\|_{X}.
\end{equation}
If
\begin{equation}\label{coarse-est}
\big\|\mathbf{e}_{\rm c}^{(\nu)}-\widetilde{\mathbf{e}}_{\rm c}^{(\nu)}\big\|_{A_{\rm c}}\leq\varepsilon_{\nu}\big\|\mathbf{e}_{\rm c}^{(\nu)}\big\|_{A_{\rm c}}
\end{equation}
for some $\varepsilon_{\nu}\in(0,1)$, then
\begin{equation}\label{iTG-est-1}
\big\|\mathbf{u}-\mathbf{u}_{\rm\s ITG}^{(\nu)}\big\|_{X}\leq\sqrt{\rho_{\nu}^{2}+\varepsilon_{\nu}^{2}(\tau_{\nu}^{2}-\rho_{\nu}^{2})}\,\big\|\mathbf{u}-\mathbf{u}^{(0)}\big\|_{X}.
\end{equation}
\end{theorem}

\begin{proof}
The condition~\eqref{coarse-est} implies
\begin{equation}\label{2er-est}
-2\widetilde{\mathbf{e}}_{\rm c}^{{(\nu)}^{T}}\mathbf{r}_{\rm c}^{(\nu)}+\big\|\widetilde{\mathbf{e}}_{\rm c}^{(\nu)}\big\|_{A_{\rm c}}^{2}\leq(\varepsilon_{\nu}^{2}-1)\big\|\mathbf{e}_{\rm c}^{(\nu)}\big\|_{A_{\rm c}}^{2}.
\end{equation}
Due to
\begin{displaymath}
\mathbf{u}-\mathbf{u}_{\rm\s ITG}^{(\nu)}=\mathbf{u}-\mathbf{u}^{(\nu)}-P_{\s X}\widetilde{\mathbf{e}}_{\rm c}^{(\nu)},
\end{displaymath}
it follows that
\begin{align*}
\big\|\mathbf{u}-\mathbf{u}_{\rm\s ITG}^{(\nu)}\big\|_{X}^{2}&=\big(\mathbf{u}-\mathbf{u}^{(\nu)}-P_{\s X}\widetilde{\mathbf{e}}_{\rm c}^{(\nu)}\big)^{T}X\big(\mathbf{u}-\mathbf{u}^{(\nu)}-P_{\s X}\widetilde{\mathbf{e}}_{\rm c}^{(\nu)}\big)\\
&=\big\|\mathbf{u}-\mathbf{u}^{(\nu)}\big\|_{X}^{2}-2\widetilde{\mathbf{e}}_{\rm c}^{{(\nu)}^{T}}P_{\s X}^{T}X\big(\mathbf{u}-\mathbf{u}^{(\nu)}\big)+\big\|\widetilde{\mathbf{e}}_{\rm c}^{(\nu)}\big\|_{A_{\rm c}}^{2}\\
&=\big\|\mathbf{u}-\mathbf{u}^{(\nu)}\big\|_{X}^{2}-2\widetilde{\mathbf{e}}_{\rm c}^{{(\nu)}^{T}}RA\big(\mathbf{u}-\mathbf{u}^{(\nu)}\big)+\big\|\widetilde{\mathbf{e}}_{\rm c}^{(\nu)}\big\|_{A_{\rm c}}^{2}\\
&=\big\|\mathbf{u}-\mathbf{u}^{(\nu)}\big\|_{X}^{2}-2\widetilde{\mathbf{e}}_{\rm c}^{{(\nu)}^{T}}\mathbf{r}_{\rm c}^{(\nu)}+\big\|\widetilde{\mathbf{e}}_{\rm c}^{(\nu)}\big\|_{A_{\rm c}}^{2}.
\end{align*}
By~\eqref{2er-est}, we have
\begin{align*}
\big\|\mathbf{u}-\mathbf{u}_{\rm\s ITG}^{(\nu)}\big\|_{X}^{2}&\leq\big\|\mathbf{u}-\mathbf{u}^{(\nu)}\big\|_{X}^{2}-(1-\varepsilon_{\nu}^{2})\big\|\mathbf{e}_{\rm c}^{(\nu)}\big\|_{A_{\rm c}}^{2}\\
&=\big\|\mathbf{u}-\mathbf{u}^{(\nu)}\big\|_{X}^{2}-(1-\varepsilon_{\nu}^{2})\big\|A_{\rm c}^{-1}\mathbf{r}_{\rm c}^{(\nu)}\big\|_{A_{\rm c}}^{2}\\
&=\big\|\mathbf{u}-\mathbf{u}^{(\nu)}\big\|_{X}^{2}-(1-\varepsilon_{\nu}^{2})\big\|A_{\rm c}^{-1}RA\big(\mathbf{u}-\mathbf{u}^{(\nu)}\big)\big\|_{A_{\rm c}}^{2}\\
&=\big\|\mathbf{u}-\mathbf{u}^{(\nu)}\big\|_{X}^{2}-(1-\varepsilon_{\nu}^{2})\big\|RA\big(\mathbf{u}-\mathbf{u}^{(\nu)}\big)\big\|_{A_{\rm c}^{-1}}^{2}\\
&=\big\|\mathbf{u}-\mathbf{u}^{(\nu)}\big\|_{X}^{2}-(1-\varepsilon_{\nu}^{2})\big\|P_{\s X}^{T}X\big(\mathbf{u}-\mathbf{u}^{(\nu)}\big)\big\|_{A_{\rm c}^{-1}}^{2}\\
&=\big\|\mathbf{u}-\mathbf{u}^{(\nu)}\big\|_{X}^{2}-(1-\varepsilon_{\nu}^{2})\big(\mathbf{u}-\mathbf{u}^{(\nu)}\big)^{T}X^{\frac{1}{2}}\Pi_{\s X}X^{\frac{1}{2}}\big(\mathbf{u}-\mathbf{u}^{(\nu)}\big),
\end{align*}
where
\begin{displaymath}
\Pi_{\s X}=X^{\frac{1}{2}}P_{\s X}A_{\rm c}^{-1}P_{\s X}^{T}X^{\frac{1}{2}}.
\end{displaymath}
Hence,
\begin{equation}\label{iTG-est-0}
\big\|\mathbf{u}-\mathbf{u}_{\rm\s ITG}^{(\nu)}\big\|_{X}^{2}\leq\big(\mathbf{u}-\mathbf{u}^{(\nu)}\big)^{T}X^{\frac{1}{2}}\big(I-(1-\varepsilon_{\nu}^{2})\Pi_{\s X}\big)X^{\frac{1}{2}}\big(\mathbf{u}-\mathbf{u}^{(\nu)}\big).
\end{equation}

Note that $\Pi_{\s X}$ is an $L^{2}$-orthogonal projector, i.e., $\Pi_{\s X}^{T}=\Pi_{\s X}=\Pi_{\s X}^{2}$. We then have
\begin{align*}
\big(\mathbf{u}-\mathbf{u}^{(\nu)}\big)^{T}X^{\frac{1}{2}}(I-\Pi_{\s X})X^{\frac{1}{2}}\big(\mathbf{u}-\mathbf{u}^{(\nu)}\big)&=\big\|(I-\Pi_{\s X})X^{\frac{1}{2}}\big(\mathbf{u}-\mathbf{u}^{(\nu)}\big)\big\|_{2}^{2}\\
&=\big\|\big(I-P_{\s X}A_{\rm c}^{-1}P_{\s X}^{T}X\big)\big(\mathbf{u}-\mathbf{u}^{(\nu)}\big)\big\|_{X}^{2}\\
&=\big\|\mathbf{u}-\mathbf{u}_{\rm\s TG}^{(\nu)}\big\|_{X}^{2},
\end{align*}
which, together with~\eqref{TG-est}, yields
\begin{equation}\label{TG-est-0}
\big(\mathbf{u}-\mathbf{u}^{(\nu)}\big)^{T}X^{\frac{1}{2}}(I-\Pi_{\s X})X^{\frac{1}{2}}\big(\mathbf{u}-\mathbf{u}^{(\nu)}\big)\leq\rho_{\nu}^{2}\big\|\mathbf{u}-\mathbf{u}^{(0)}\big\|_{X}^{2}.
\end{equation}
Using~\eqref{smoothing}, \eqref{iTG-est-0}, and~\eqref{TG-est-0}, we obtain
\begin{align*}
\big\|\mathbf{u}-\mathbf{u}_{\rm\s ITG}^{(\nu)}\big\|_{X}^{2}&\leq(1-\varepsilon_{\nu}^{2})\rho_{\nu}^{2}\big\|\mathbf{u}-\mathbf{u}^{(0)}\big\|_{X}^{2}+\varepsilon_{\nu}^{2}\big\|\mathbf{u}-\mathbf{u}^{(\nu)}\big\|_{X}^{2}\\
&\leq\big((1-\varepsilon_{\nu}^{2})\rho_{\nu}^{2}+\varepsilon_{\nu}^{2}\tau_{\nu}^{2}\big)\big\|\mathbf{u}-\mathbf{u}^{(0)}\big\|_{X}^{2},
\end{align*}
which gives~\eqref{iTG-est-1}.
\end{proof}

\begin{remark}\label{rmk:uniform}
In general, $\rho_{\nu}<\tau_{\nu}$, since the combination of smoothing and coarse-grid correction reduces the error more effectively than smoothing alone. Then
\begin{displaymath}
\sqrt{\rho_{\nu}^{2}+\varepsilon_{\nu}^{2}(\tau_{\nu}^{2}-\rho_{\nu}^{2})}<\tau_{\nu}\leq 1 \quad \forall\,\varepsilon_{\nu}\in(0,1),
\end{displaymath}
which means that the resulting inexact two-grid method is convergent for any accuracy parameter $\varepsilon_{\nu}\in(0,1)$. Furthermore, if $\rho_{\nu}\leq C$ and $\varepsilon_{\nu}\leq\varepsilon$ for some constants $C\in(0,1)$ and $\varepsilon\in(0,1)$, then
\begin{displaymath}
\sqrt{\rho_{\nu}^{2}+\varepsilon_{\nu}^{2}(\tau_{\nu}^{2}-\rho_{\nu}^{2})}\leq\sqrt{\rho_{\nu}^{2}+\varepsilon^{2}(1-\rho_{\nu}^{2})}\leq\sqrt{\varepsilon^{2}+(1-\varepsilon^{2})C^{2}}<1,
\end{displaymath}
that is, the factor $\sqrt{\rho_{\nu}^{2}+\varepsilon_{\nu}^{2}(\tau_{\nu}^{2}-\rho_{\nu}^{2})}$ has a uniform upper bound.
\end{remark}

As a corollary of Theorem~\ref{thm:iTG}, we have the following result.

\begin{corollary}\label{cor:iTG}
Under the conditions~\eqref{TG-est} and~\eqref{coarse-est}, if
\begin{equation}\label{utg-uv}
\big\|\mathbf{u}_{\rm\s TG}^{(\nu)}-\mathbf{u}^{(\nu)}\big\|_{X}\leq\eta_{\nu}\big\|\mathbf{u}-\mathbf{u}^{(0)}\big\|_{X}
\end{equation}
for some $\eta_{\nu}\in(0,1)$, then
\begin{equation}\label{iTG-est-2}
\big\|\mathbf{u}-\mathbf{u}_{\rm\s ITG}^{(\nu)}\big\|_{X}\leq\sqrt{\rho_{\nu}^{2}+\varepsilon_{\nu}^{2}\eta_{\nu}^{2}}\,\big\|\mathbf{u}-\mathbf{u}^{(0)}\big\|_{X}.
\end{equation}
\end{corollary}

\begin{proof}
In view of~\eqref{Ae=r} and~\eqref{utg}, we have
\begin{align*}
\big(\mathbf{u}-\mathbf{u}_{\rm\s TG}^{(\nu)}\big)^{T}X\big(\mathbf{u}_{\rm\s TG}^{(\nu)}-\mathbf{u}^{(\nu)}\big)&=\big(\mathbf{u}-\mathbf{u}^{(\nu)}-P_{\s X}\mathbf{e}_{\rm c}^{(\nu)}\big)^{T}XP_{\s X}\mathbf{e}_{\rm c}^{(\nu)}\\
&=\big(\mathbf{u}-\mathbf{u}^{(\nu)}-P_{\s X}A_{\rm c}^{-1}\mathbf{r}_{\rm c}^{(\nu)}\big)^{T}XP_{\s X}\mathbf{e}_{\rm c}^{(\nu)}\\
&=\big(\mathbf{u}-\mathbf{u}^{(\nu)}\big)^{T}\big(I-P_{\s X}A_{\rm c}^{-1}RA\big)^{T}XP_{\s X}\mathbf{e}_{\rm c}^{(\nu)}\\
&=\big(\mathbf{u}-\mathbf{u}^{(\nu)}\big)^{T}\big(I-P_{\s X}A_{\rm c}^{-1}P_{\s X}^{T}X\big)^{T}XP_{\s X}\mathbf{e}_{\rm c}^{(\nu)}\\
&=\big(\mathbf{u}-\mathbf{u}^{(\nu)}\big)^{T}\big(I-XP_{\s X}(P_{\s X}^{T}XP_{\s X})^{-1}P_{\s X}^{T}\big)XP_{\s X}\mathbf{e}_{\rm c}^{(\nu)}\\
&=0,
\end{align*}
which leads to
\begin{displaymath}
\big\|\mathbf{u}-\mathbf{u}^{(\nu)}\big\|_{X}^{2}=\big\|\mathbf{u}-\mathbf{u}_{\rm\s TG}^{(\nu)}\big\|_{X}^{2}+\big\|\mathbf{u}_{\rm\s TG}^{(\nu)}-\mathbf{u}^{(\nu)}\big\|_{X}^{2}.
\end{displaymath}
This, combined with~\eqref{TG-est} and~\eqref{utg-uv}, yields
\begin{equation}\label{smoothing-rhoeta}
\big\|\mathbf{u}-\mathbf{u}^{(\nu)}\big\|_{X}\leq\sqrt{\rho_{\nu}^{2}+\eta_{\nu}^{2}}\,\big\|\mathbf{u}-\mathbf{u}^{(0)}\big\|_{X},
\end{equation}
i.e., \eqref{smoothing} holds with $\tau_{\nu}=\sqrt{\rho_{\nu}^{2}+\eta_{\nu}^{2}}$. Note that~\eqref{iTG-est-1} remains valid even if $\tau_{\nu}>1$. Applying~\eqref{iTG-est-1} gives~\eqref{iTG-est-2}.
\end{proof}

\begin{remark}
Observe that~\eqref{iTG-est-2} is of interest mainly when $\eta_{\nu}\leq\sqrt{1-\rho_{\nu}^{2}}$, in which case $\sqrt{\rho_{\nu}^{2}+\varepsilon_{\nu}^{2}\eta_{\nu}^{2}}<1$ for any $\varepsilon_{\nu}\in(0,1)$. If $\eta_{\nu}>\sqrt{1-\rho_{\nu}^{2}}$, then the factor $\sqrt{\rho_{\nu}^{2}+\eta_{\nu}^{2}}$ in~\eqref{smoothing-rhoeta} is strictly greater than one, which implies that the smoothing process may amplify the error. However, even in the latter case, $\sqrt{\rho_{\nu}^{2}+\varepsilon_{\nu}^{2}\eta_{\nu}^{2}}<1$ may still hold, for instance, when $\varepsilon_{\nu}$ is small.
\end{remark}

\begin{remark}
Clearly, \eqref{smoothing}, \eqref{TG-est}, and~\eqref{utg-uv} characterize the pairwise distances, in the $X$-norm, among $\mathbf{u}$, $\mathbf{u}^{(\nu)}$, and $\mathbf{u}_{\rm\s TG}^{(\nu)}$. In addition, the proof of Corollary~\ref{cor:iTG} shows that $\mathbf{u}-\mathbf{u}_{\rm\s TG}^{(\nu)}$ and $\mathbf{u}_{\rm\s TG}^{(\nu)}-\mathbf{u}^{(\nu)}$ are orthogonal with respect to the $X$-inner product. In the two-dimensional case, these distance and orthogonality relations are illustrated in Figure~\ref{fig:circle}.
\end{remark}

\begin{figure}[!htbp]
	
\vskip -0.3cm

\centering

\includegraphics[width=0.46\textwidth]{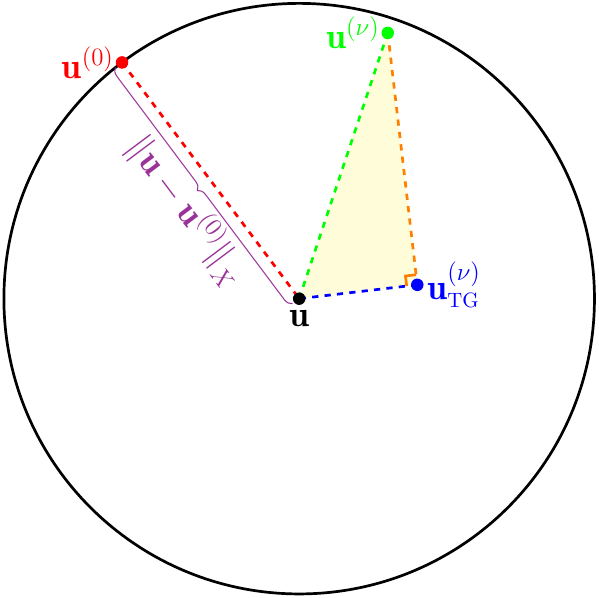}

\vskip -0.2cm

\caption{Two-dimensional illustration of~\eqref{smoothing}, \eqref{TG-est}, and~\eqref{utg-uv}.}

\label{fig:circle}

\vskip -0.25cm

\end{figure}

\begin{remark}
Note that the inexact two-grid theory developed in this section applies to more general systems, since the positive definiteness of $A$ is not used in the proofs of Theorem~\ref{thm:iTG} and Corollary~\ref{cor:iTG}.
\end{remark}

\section{Numerical experiments} \label{sec:numer}

Our analysis shows that $P_{\s X}=X^{-1}A^{T}R^{T}$ is a suitable choice for the prolongation matrix when the error is measured in the $X$-norm. Let $H$ and $S$ denote the symmetric and skew-symmetric parts of $A$, respectively; that is,
\begin{displaymath}
H=\frac{1}{2}(A+A^{T}) \quad \text{and} \quad S=\frac{1}{2}(A-A^{T}).
\end{displaymath}
Then
\begin{displaymath}
A=H+S.
\end{displaymath}
Since $H$ is SPD, a natural choice is $X=H$. The resulting $H$-norm generalizes the $A$-norm used in the SPD setting.

However, $P_{\s H}=H^{-1}A^{T}R^{T}$ may not be a practical choice, since computing $H^{-1}$ is often too costly. An alternative is to construct a practical prolongation by approximating $H^{-1}$ or its action on vectors. For example, one may approximate $H^{-1}$ by $\omega\diag(H)^{-1}$, where $\omega>0$ is a parameter. Substituting this approximation into the expression
\begin{equation}\label{PH}
P_{\s H}=H^{-1}(H-S)R^{T}=(I-H^{-1}S)R^{T}
\end{equation}
yields the practical prolongation matrix
\begin{equation}\label{PH-app}
\widetilde{P}_{\s H}^{\s (\omega)}=\big(I-\omega\diag(H)^{-1}S\big)R^{T}.
\end{equation}
In many applications, it suffices to approximate the action of $P_{\s H}$ on vectors, without explicitly constructing an approximation of $P_{\s H}$ itself. For any $\mathbf{v}_{\rm c}\in\mathbb{R}^{n_{\rm c}}$, we have
\begin{displaymath}
P_{\s H}\mathbf{v}_{\rm c}=R^{T}\mathbf{v}_{\rm c}-H^{-1}SR^{T}\mathbf{v}_{\rm c}.
\end{displaymath}
The term $H^{-1}SR^{T}\mathbf{v}_{\rm c}$ can be approximated by solving the SPD linear system
\begin{displaymath}
H\mathbf{x}=SR^{T}\mathbf{v}_{\rm c},
\end{displaymath}
for instance, using multigrid or conjugate gradient.

The prolongation matrix~\eqref{PH-app} can be viewed as a left multiplicative perturbation of the classical choice $R^{T}$. To compare the performance of $\widetilde{P}_{\s H}^{\s(\omega)}$ and $R^{T}$, we present a numerical example.

Consider the following convection-diffusion equation with homogeneous Dirichlet boundary conditions on the unit square $\Omega=(0,1)\times(0,1)$:
\begin{equation}\label{conv-diff}
\left\{
\begin{aligned}
-\Delta u+au_{x}+bu_{y}&=f \quad \text{in $\Omega$},\\
u&=0 \quad \text{on $\partial\Omega$},
\end{aligned}
\right.
\end{equation}
where $u=u(x,y)$, $f=f(x,y)$, and $a,b\geq 0$ are constants. The problem~\eqref{conv-diff} is discretized on a uniform grid with grid points
\begin{displaymath}
(x_{i},y_{j})=(ih,jh) \quad i,j=0,1,\ldots,m+1,
\end{displaymath}
where $h=1/(m+1)$. Specifically, the Laplacian operator is discretized using the standard five-point difference scheme, and the derivatives $u_{x}$ and $u_{y}$ are discretized using the first-order upwind difference scheme. The resulting stencil takes the form
\begin{displaymath}
\frac{1}{h^{2}}\begin{bmatrix}
& -1 &  \\
-1-2\alpha & 4+2(\alpha+\beta) & -1 \\
& -1-2\beta & 
\end{bmatrix}_{h},
\end{displaymath}
where
\begin{displaymath}
\alpha=\frac{ah}{2} \quad \text{and} \quad \beta=\frac{bh}{2}.
\end{displaymath}
Let
\begin{align*}
T_{1}&={\rm tridiag}(-1-2\alpha,\,4+2(\alpha+\beta),\,-1)\in\mathbb{R}^{m\times m},\\
T_{2}&={\rm tridiag}(0,\,0,\,-1)\in\mathbb{R}^{m\times m}.
\end{align*}
In our experiments, all functions and their approximations evaluated at $(x_{i},y_{j})$ are ordered lexicographically. This yields the linear system
\begin{displaymath}
A\mathbf{u}=\mathbf{f},
\end{displaymath}
where
\begin{displaymath}
A=I_{m}\otimes T_{1}+\big(T_{2}+(1+2\beta)T_{2}^{T}\big)\otimes I_{m}
\end{displaymath}
and $\mathbf{f}\in\mathbb{R}^{m^{2}}$ is formed by ordering the sequence $\big\{h^{2}f(x_{i},y_{j})\big\}_{i,j=1}^{m}$ lexicographically. Here, $\otimes$ denotes the Kronecker product.

Let $m$ be odd and $\mathcal{I}=\{1,3,\ldots,m\}$. The coarse points are chosen as $(x_{i},y_{j})$ with indices $(i,j)\in\mathcal{I}\times\mathcal{I}$; see the black solid circles in Figure~\ref{fig:transfer}. The main experimental settings are as follows:
\begin{itemize}[leftmargin=0.8cm]

\item $f$ is chosen as the zero function;

\item $\mathbf{u}^{(0)}$ is generated randomly and then fixed;

\item $M_{k}=1.5\diag(A)$ for $k=1,2,\ldots,\nu$, where $\nu=4$;

\item $R=R_{0}\otimes R_{0}$, where
\begin{displaymath}
R_{0}=\frac{1}{4}\begin{pmatrix}
2 & 1 &   &   &   &   &   &   & \\
& 1 & 2 & 1 &   &   &   &   & \\
&   &   &   & \ddots &   &   &   & \\
&   &   &   &        & 1 & 2 & 1 & \\
&   &   &   &        &   &   & 1 & 2
\end{pmatrix}\in\mathbb{R}^{\frac{m+1}{2}\times m}
\end{displaymath}
is based on the restriction stencil
\begin{displaymath}
\frac{1}{16}\begin{bmatrix}
1 & 2 & 1 \\
2 & 4 & 2 \\
1 & 2 & 1 \\
\end{bmatrix}_{h}^{2h};
\end{displaymath}

\item $\widetilde{P}_{\s H}^{\s(\omega)}$ is given by~\eqref{PH-app}.

\end{itemize}

\begin{figure}[!htbp]

\centering

\includegraphics[width=0.8\textwidth]{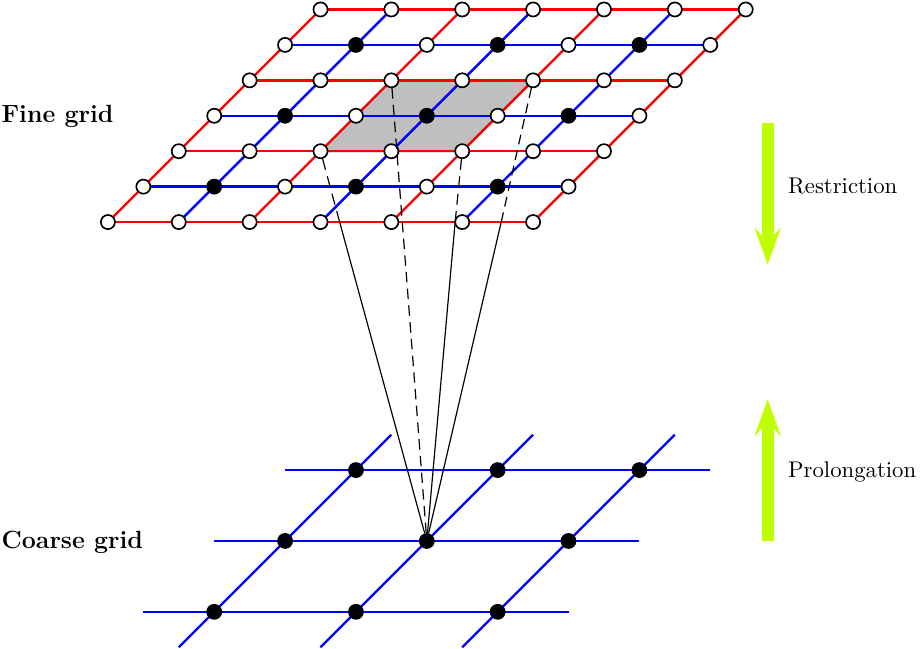}

\vskip -0.1cm

\caption{Transfer between fine and coarse grids.}

\label{fig:transfer}

\vskip -0.1cm

\end{figure}

The numerical results in Table~\ref{tab:comp} show that:
\begin{itemize}[leftmargin=0.8cm]

\item In the convection-dominated regime (when $a$ and $b$ are large), the prolongation matrix $\widetilde{P}_{\s H}^{\s(2.5)}$ yields smaller convergence factors than the classical choice $R^{T}$;

\item The two-grid method with $\widetilde{P}_{\s H}^{\s(2.5)}$ exhibits greater robustness (its convergence factor varies less as $a$ increases) than that with $R^{T}$.

\end{itemize}

\renewcommand{\arraystretch}{1.08}

\begin{table}[h!]

\begin{center}

\begin{tabular}{c|c|cccccc}
\hline
\multirow{2}*{} &\multirow{2}*{} & \multicolumn{6}{c}{PDE coefficients $(a=b)$} \\ \cline{3-8}
& & $a=10^{1}$ & $a=10^{2}$ & $a=10^{3}$ & $a=10^{4}$ & $a=10^{5}$ & $a=10^{6}$ \\
\hline
\multirow{2}*{$h=1/128$} & $R^{T}$ & $0.20$ & $0.14$ & $0.19$ & $0.35$ & $0.37$ & $0.37$ \\
& $\widetilde{P}_{\s H}^{\s(2.5)}$ & $0.20$ & $0.18$ & $0.10$ & $0.14$ & $0.15$ & $0.15$ \\
\hline
\multirow{2}*{$h=1/256$} & $R^{T}$ & $0.20$ & $0.18$ & $0.09$ & $0.33$ & $0.37$ & $0.38$ \\
& $\widetilde{P}_{\s H}^{\s(2.5)}$ & $0.20$ & $0.19$ & $0.10$ & $0.13$ & $0.16$ & $0.16$ \\
\hline
\multirow{2}*{$h=1/512$} & $R^{T}$ & $0.20$ & $0.19$ & $0.07$ & $0.29$ & $0.37$ & $0.38$ \\
& $\widetilde{P}_{\s H}^{\s(2.5)}$ & $0.20$ & $0.20$ & $0.14$ & $0.12$ & $0.16$ & $0.16$ \\
\hline
\multirow{2}*{$h=1/1024$} & $R^{T}$ & $0.20$ & $0.20$ & $0.12$ & $0.23$ & $0.36$ & $0.38$ \\
& $\widetilde{P}_{\s H}^{\s(2.5)}$ & $0.20$ & $0.20$ & $0.17$ & $0.11$ & $0.15$ & $0.16$ \\
\hline
\multirow{2}*{$h=1/2048$} & $R^{T}$ & $0.20$ & $0.20$ & $0.17$ & $0.12$ & $0.35$ & $0.38$ \\
& $\widetilde{P}_{\s H}^{\s(2.5)}$ & $0.20$ & $0.20$ & $0.19$ & $0.10$ & $0.14$ & $0.16$ \\
\hline
\end{tabular}
	
\end{center}

\bigskip

\caption{Comparison of two-grid asymptotic convergence factors.}

\label{tab:comp}

\vskip -0.5cm

\end{table}

\begin{remark}
The prolongation matrix $\widetilde{P}_{\s H}^{\s(\omega)}$ is obtained by replacing $H^{-1}$ in~\eqref{PH} with $\omega\diag(H)^{-1}$, which is in fact a scalar matrix. In our numerical tests, we take $\omega=2.5$. A better prolongation may be obtained by optimizing $\omega$. Other strategies for approximating $H^{-1}$ are also possible.
\end{remark}

\begin{remark}
Since $S$ is skew-symmetric and $\diag(H)=\big(4+2(\alpha+\beta)\big)I_{m^{2}}$, every eigenvalue of $I-\omega\diag(H)^{-1}S$ has real part one; hence this matrix is nonsingular. Therefore, $\widetilde{P}_{\s H}^{\s (\omega)}$ has full column rank whenever $R$ has full row rank.
\end{remark}

\section{Conclusions} \label{sec:con}

In this paper, we present a convergence analysis of two-grid methods for nonsymmetric positive definite systems. For the case of an exact coarse solver, we derive an identity for the two-grid convergence factor in a smoother-induced norm. The identity is used to analyze the optimality of the restriction matrix and the influence of its row space on the convergence factor. More generally, we develop a convergence theory for two-grid methods with inexact coarse solvers. The theory is established in a generic norm and applies to more general systems. Our numerical results show that the prolongation matrix motivated by our theory can outperform the classical choice in some cases. This offers new insights into the design of multigrid methods.

\section*{Acknowledgments}

The author would like to thank Prof.~Chen-Song Zhang (AMSS, Chinese Academy of Sciences) and Prof.~Xiaozhe Hu (Tufts University) for their valuable comments and suggestions. This work was partially supported by the National Natural Science Foundation of China (Grant No.~12401479), the Natural Science Foundation of Jiangsu Province (Grant No.~BK20241257), and the Start-up Research Fund of Southeast University (Grant No.~RF1028623372).

\bibliographystyle{amsplain}

\end{document}